\documentclass{elsarticle}

\usepackage{amsfonts}
\usepackage{graphicx}
\usepackage{amsopn}
\usepackage{amssymb}

\usepackage{amsmath}
\usepackage{amsthm}
\usepackage{cleveref}

\newtheorem{theorem}{Theorem}
\newtheorem{lemma}[theorem]{Lemma}
\newtheorem{remark}[theorem]{Remark}
\newtheorem{proposition}[theorem]{Proposition}

\newcommand{\boldzero}{\mathbf{0}}

\usepackage{enumitem}
\setlist[enumerate]{leftmargin=.5in}
\setlist[itemize]{leftmargin=.5in}

\title{
    A geometric criterion
    on the equality between BKK bound and intersection index
}
\author[1]{Tianran Chen}
\ead{ti@nranchen.org} 
\affiliation[1]{
    organization={Department of Mathematics, Auburn University at Montgomery},
    addressline={Montgomery, AL USA}
}

\DeclareMathOperator{\Newt}{Newt}
\DeclareMathOperator{\Mvol}{MVol}
\DeclareMathOperator{\Init}{init}
\DeclareMathOperator{\Span}{span}
\DeclareMathOperator{\conv}{conv}

\newcommand{\C}{\mathbb{C}}
\newcommand{\R}{\mathbb{R}}
\newcommand{\Z}{\mathbb{Z}}
\newcommand{\bolda}{\mathbf{a}}
\newcommand{\bolde}{\mathbf{e}}
\newcommand{\boldx}{\mathbf{x}}
\newcommand{\boldv}{\mathbf{v}}

\begin{document}


\begin{abstract}
    The Bernshtein-Kushnirenko-Khovanskii theorem provides a generic root count
    for system of Laurent polynomials in terms of the mixed volume of their Newton polytopes
    (i.e., the BKK bound).
    A recent and far-reaching generalization of this theorem is the study of
    birationally invariant intersection index by Kaveh and Khovanskii.
    This short note establishes a simple geometric condition
    on the equality between the BKK bound and the intersection index
    for a system of vector spaces of Laurent polynomials.
    Applying this, we show that the intersection index
    for the algebraic Kuramoto equations equals their BKK bound.
\end{abstract}



\maketitle

\section{Introduction}\label{sec:intro}

The Bernshtein-Kushinirenko-Khovanskii 
theorem~\cite{Bernshtein1975Number,Khovanskii1978Newton,Kushnirenko1976Polyedres,Kushnirenko1975Newton,Kushnirenko1976Newton}
relates the root counting problem for systems of polynomial equations
and the theory of convex bodies.
In particular, it states that the generic (and hence maximum)
number of isolated solutions a system of Laurent polynomial equations has
in the algebraic torus $(\C^*)^n = (\C \setminus \{0\})^n$ equals
the mixed volume of their Newton polytopes.
This is the Bernshtein-Kushnirenko-Khovanskii (BKK) bound.

Recently, a far-reaching generalization of this theorem is developed in a series 
of works~\cite{KavehKhovanskii2008Convex,KavehKhovanskii2012Newton,KavehKhovanskii2010Mixed}
by K. Kaveh and A. Khovanskii where the root counting question is considered
for much more general spaces of rational functions.
Given an irreducible $n$-dimensional complex algebraic variety $X$
and $n$-tuple of finite dimensional vector spaces $(L_1,\dots,L_n)$
of rational functions on $X$,
for generic elements $f_i \in L_i$ for $i=1,\dots,n$,
the number of common solutions a system $f_1 = \dots = f_n = 0$
has in $X$ is closely related to the
geometry of the Newton-Okunkov bodies associated with $L_1,\dots,L_n$.
This generic root count is given the name
\emph{birationally invariant intersection index}
(or simply, intersection index).

In this short note we establish a simple geometric criterion
on the equality between the BKK bound
and the more refined intersection index,
even when the space of functions
are not generated by monomials. 


\section{Preliminaries}\label{sec:prelim}

For $\boldx = (x_1,\dots,x_n)$ and 
$\bolda = \begin{bmatrix} a_1 & \dots & a_n \end{bmatrix}^\top \in \Z^n$,
we use the notation $\boldx^{\bolda} = x_1^{a_1} \cdots x_n^{a_n}$
for the \emph{Laurent monomial}.
For a \emph{Laurent polynomial} 
$f(\boldx) = \sum_{\bolda \in S} c_\bolda \boldx^{\bolda}$,
$\Newt(f) := \conv(S)$ is its \emph{Newton polytope}.
With respect to a vector $\boldv \in \R^n$, 
its \emph{initial form} $\Init_{\boldv}(f)$ is the expression 
$\sum_{\bolda \in (S)_{\boldv}} c_{\bolda} \boldx^{\bolda}$
where $(S)_{\boldv} \subset S$ is the subset on which the linear functional
$\langle \boldv \,,\, \cdot \rangle$ is minimized.
$\C[x_1^\pm,\dots,x_n^\pm]$ denotes the set of all Laurent polynomials in $x_1,\dots,x_n$.

The natural space to study this root counting question is 
the \emph{algebraic torus} $(\C^*)^n = (\C \setminus \{0\})^n$.
While the root count of a Laurent polynomial system can vary greatly depending on the coefficients,
for ``generic'' coefficients the $(\C^*)^n$-root count  
remains a constant and only depends on its monomial structure.
D. Bernshtein showed this constant is precisely the mixed volume\footnote{%
    Here, we follow the convention that
    the mixed volume $\Mvol(P_1,\ldots,P_n)$
    of $n$ convex polytopes $P_1,\ldots,P_n \subset \R^n$
    is the coefficient of the mixed term $\lambda_1 \cdots \lambda_n$ 
    in the homogenous polynomial
    $\operatorname{Vol}(\lambda_1 P_1 + \cdots + \lambda_n P_n)$~\cite{Minkowski1911Theorie},
    where $\operatorname{Vol}$ is the Euclidean volume form.
}
of their Newton polytopes \cite{Bernshtein1975Number}. 
This is also an upper bound on the number of isolated $\C^*$-zeros
such a Laurent polynomial system could have,
and it is known as the
\emph{Bernshtein-Kushnirenko-Khovanskii} (BKK) bound,
after a circle of closely related works by
Bernshtein~\cite{Bernshtein1975Number},
Kushnirenko~\cite{Kushnirenko1976Polyedres,Kushnirenko1975Newton,Kushnirenko1976Newton}, and
Khovanskii~\cite{Khovanskii1978Newton}.
The arguably more important part of Bernshtein's paper
\cite{Bernshtein1975Number} is his second theorem:

\begin{theorem}[Bernshtein 1975~\cite{Bernshtein1975Number}]\label{thm:bernshtein-b}
    Given a Laurent polynomial system $\mathbf{f} = (f_1,\dots,f_n)$ 
    in $\boldx = (x_1,\dots,x_n)$,
    if for all $\boldzero \ne \boldv \in \R^n$,
    the initial system $\Init_{\boldv}(\mathbf{f})$ has no zero in $(\C^*)^n$,
    then all zeros of $\mathbf{f}$ in $(\C^*)^n$ are isolated,
    and the total number, counting multiplicity, is the mixed volume
    $\Mvol(\Newt(f_1),\dots,\Newt(f_n))$.
\end{theorem}


Systems satisfying this condition
are said to be \emph{Bernshtein-general}.
Bernshtein showed Bernshtein-generalness
hold for generic choices of coefficients:

\begin{lemma}[Bernshtein 1975~\cite{Bernshtein1975Number}]\label{lem:bernshtein}
    Given a Laurent polynomial system $\mathbf{f} = (f_1,\dots,f_n)$ 
    in $\boldx = (x_1,\dots,x_n)$,
    there is a nonempty Zariski-open set 
    of the coefficients for which 
    the initial system $\Init_\boldv (\mathbf{f})$ has no solution in $(\C^*)^n$
    for any nonzero vector $\boldv \in \R^n$.
\end{lemma}


These results have been generalized into
the theory of
\emph{birationally invariant intersection index}
\cite{KavehKhovanskii2012Newton,KavehKhovanskii2010Mixed}.
Instead of considering
generic linear combinations of Laurent monomials,
one could start with $\C$-vector spaces
$L_1,\dots,L_n$ each spanned by
finitely many rational functions 
on an irreducible toric variety $X$.
Then for generic choices of functions $f_1 \in L_1, \dots, f_n \in L_n$,
the number of common zeros $(f_1,\dots,f_n)$ has
(away from base locus of the system)
in $X$ is a constant that is independent of the choices.
This is the \emph{intersection index} of $L_1,\dots,L_n$
and is denoted by $[L_1,\dots,L_n]$.
This grand theory relates the root counting problem to the geometric properties
of \emph{Newton-Okunkov bodies},
and the BKK bound is thus a special case of this intersection index
in the situations where each $L_i$ is spanned by Laurent monomials.
In the following, we extend the BKK bound to certain cases where
each $L_i$ is spanned by Laurent polynomials.

\section{The main theorem}\label{sec:main}

The goal here is to show the equality of the
intersection index and the BKK bound
under a simple geometric condition
and thereby generalize the theory of BKK bound.
We focus on the cases where $X = (\C^*)^n$ and
$L_1,\dots,L_n$ are spanned by finitely many
Laurent polynomials. 
That is,
\begin{equation}
    L_i = \Span_{\C} \{ P_{ij} \}_{j = 1}^{m_i}
\end{equation}
where each $P_{ij} \in \C[x_1^{\pm},\dots,x_n^{\pm}]$ is nonzero
and $m_i \in \Z^+$.
This setup is a generalization of the situation in~\Cref{thm:bernshtein-b}
where each $L_i$ is only spanned by a set of Laurent monomials.
That is, if each $P_{ij}$ is a Laurent monomial,
then $[\, L_1, \dots, L_n \,]$ is exactly the BKK bound.
The main result of this note is a geometric condition
under which the BKK bound remains sharp 
even when each $P_{ij}$ is a Laurent polynomial.

A generic element $f_i \in L_i$, is a Laurent polynomial
$f_i = \sum_{j=1}^{m_i} c_{ij} P_{ij}$ with generic choice of the coefficients
$c_{i1},\dots,c_{im_i}$.
It is easy to see that among the terms within such a generic element,
there are no cancellations and consequently
$\Newt(f_i) = \conv \left( \cup_{j=1}^{m_i} \Newt(P_{ij}) \right)$.
It is therefore reasonable to use the notation
\[
    \Newt(L_i) := \conv \left( \cup_{j=1}^{m_i} \Newt(P_{ij}) \right).
\]
By Bernshtein's Theorem \cite{Bernshtein1975Number},
\[
    [\, L_1,\ldots, L_n \,] \le 
    \Mvol(\Newt(L_1),\ldots,\Newt(L_n)).
\]
The equality does not hold in general.
The main result of this note is a sufficient condition
on the equality between the two root counts
stated in terms of the ``exposure'' of each Newton polytope
$\Newt(P_{ij})$ on the boundary of $\Newt(L_i)$.
This condition is, intentionally, chosen to only rely on
the geometric information that can be obtained from
the individual Newton polytopes $\Newt(L_1),\ldots,\Newt(L_n)$
and does not require information from
any of their mixed subdivisions.
In other words,
the condition to be established here
is purely polytopal, not tropical.

\begin{theorem}\label{thm:main}
    Let $L_1,\dots,L_n$ be vector spaces of rational functions
    with $L_i = \Span_{\C} \{ P_{ij} \}_{j = 1}^{m_i}$
    where each $P_{ij} \in \C[x_1^{\pm},\dots,x_n^{\pm}]$ and $m_i \in \Z^+$
    as described above.
    If for each $i=1,\dots,n$, we have
    \begin{enumerate}
        \item $\dim(\Newt(L_i)) = n$,
        \item functions in $L_i$ have no common zeros in $(\C^*)^n$,
        \item every positive-dimensional proper faces of $\Newt(L_i)$
            intersect $\Newt(P_{ij})$ at no more than one point for each $j=1,\dots,m_i$,
    \end{enumerate}
    then
    \begin{equation}
        [\, L_1,\dots,L_n \,] =
        \Mvol\,(\, \Newt(L_1), \dots, \Newt(L_n) \,).
    \end{equation}
\end{theorem}

\begin{proof}
    Let $f_1,\dots,f_n$ be generic elements in $L_1,\dots,L_n$ respectively,
    i.e., $f_i = \sum_{j=1}^{m_i} c_{ij} P_{ij}$
    for generic coefficients $\{ c_{ij} \}$.
    Since it is also assumed that Laurent polynomials in each $L_i$
    have no common roots in $(\C^*)^n$,
    the common root count of the system
    $\mathbf{f} = (f_1,\dots,f_n)$ in $(\C^*)^n$
    equals the intersection index $[\, L_1,\dots,L_n \,]$.
    It is therefore sufficient to show the root count of $\mathbf{f}$ in $(\C^*)^n$
    matches the BKK bound, i.e., 
    $\mathbf{f}$ satisfies the conditions in~\Cref{thm:bernshtein-b}.
    
    Let $\boldv \in \R^n$ be a nonzero vector such that $\Init_{\boldv}(\mathbf{f})$
    does not contain a unit
    (i.e., no component of $\mathbf{f}$ is a single Laurent monomial term).
    Since $\Newt(f_i) = \Newt(L_i)$ is assumed to be full-dimensional for $i=1,\dots,n$,
    $\boldv$ must be a common inner normal vector for $n$ proper positive dimensional 
    faces $F_1,\dots,F_n$ of $\Newt(f_1),\dots,\Newt(f_n)$ respectively.
    
    For each $i=1,\dots,n$, let $A_{ij} = F_i \cap \Newt(P_{ij})$, 
    then, by assumption, each $A_{ij}$ contains at most one point.
    Without loss of generality, after re-indexing $P_{ij}$'s,
    we can assume that for a fixed $i$, 
    $A_{ij} = \{ \bolda_{ij} \}$ for $j = 1,\dots,m_i'$
    and $A_{ij} = \varnothing$ for $j=m_i'+1,\dots,m_i$
    where $m_i' \in \Z^+$ and $m_i' \le m_i$
    (since $F_i \cap \Newt(P_{ij})$ may be empty for some $j$).
    With this definition,
    $\{ \bolda_{i,1}, \dots, \bolda_{i,m_i'} \} = \bigcup_{j=1}^{m_i} A_{ij}$,
    and consequently,
    \[
        \Init_{\boldv} (f_i) \in 
        \Span_{\C} \{ 
            \boldx^{\bolda_{i1}},
            \dots,
            \boldx^{\bolda_{im_i'}}
        \}.
    \]
    Moreover, the set of coefficients is a subset of the
    coefficients in $f_i$. Indeed,
    \[
        \Init_{\boldv}  (f_i) =
        \sum_{j=1}^{m_i'} c_{ij} \boldx^{\bolda_{i,j}}
    \]
    with the exponent vectors $\bolda_{i1},\dots,\bolda_{im_i'}$
    all lie in a proper face of $\Newt(L_i)$
    and the coefficients $c_{ij}$'s being independent from one another.
    By \cref{lem:bernshtein}, there exists a nonempty Zariski open set
    in the coefficient space $\{ c_{ij} \}$ for which the initial system
    $\Init_{\boldv}(\mathbf{f})$ has no solution in $(\C^*)^n$.
    
    Note that there are only finitely many distinct initial systems for $\mathbf{f}$.
    By taking the intersection of a finite number of nonempty Zariski open set,
    we can see that there remains a nonempty Zariski open set in the
    coefficient space $\{ c_{ij} \}$ such that for all choices in this set,
    $\Init_{\boldv}(\mathbf{f})$ either contains a unit
    or has no solution in $(\C^*)^n$ for any nonzero vector $\boldv \in \R^n$.
    
    By \Cref{thm:bernshtein-b}, for generic choices of the coefficients
    $\{ c_{ij} \}$, the BKK bound for $\mathbf{f}$ is exact,
    i.e., the common root count in $(\C^*)^n$ for this system 
    is exactly $\Mvol(\Newt(f_1),\dots,\Newt(f_n))$.
    Recall that each $f_i$ is a generic member of $L_i$.
    This shows
    \[
        [\, L_1,\dots,L_n \,] = \Mvol(\,\Newt(L_1),\dots,\Newt(L_n)\,).
        \qedhere
    \]
\end{proof}

\begin{remark}
    Condition 2 of \Cref{thm:main}
    is needed because the root count
    of a system in $(L_1,\ldots,L_n)$
    may include points at which all functions in $L_i$ vanish
    for some $i$,
    while these points are excluded
    from the intersection index
    $[\, L_1, \ldots, L_n]$.

    Neither condition 1 nor 3 are necessary.
    In this note, we focus on the problem of detecting
    the equality between the BKK bound and the intersection index
    using geometric information of the individual Newton polytopes
    in $\Newt(L_1),\ldots,\Newt(L_n)$,
    rather than the more refined tropical information.
    In \Cref{rmk: kuramoto root count},
    we highlight some more general results.
\end{remark}

\section{Generic root count of Kuramoto equations}\label{sec:kuramoto}

The Kuramoto model~\cite{Kuramoto1975Self} is a ubiquitous model for 
studying the phenomenon of spontaneous synchronization of 
networks of coupled oscillators.
It has long been known that
when simple harmonic oscillators are coupled with one another,
complicated collective behaviors emerge.
Biological examples include
pacemaker cells in the heart and
the formation of circadian rhythm in the brain.
Such models have since found important applications in
many other seemingly independent research fields.

An oscillator can be modeled as a moving point on
the complex plane circling 0.
A swarm of such points interacting with one another
thus form a network of coupled oscillators.
For weakly coupled and nearly identical oscillators,
Winfree intuited that there is
natural separation of timescales:
On the short timescale,
oscillators are approximated by their limit cycles
and thus can be represented by their phases \cite{Winfree1967Biological}.
This is derived rigorously by Kuramoto
\cite{Kuramoto1975Self} using perturbation methods.
Kuramoto singled out the simplest case with the governing equations
\begin{equation}\label{equ: generalized kuramoto}
  \dot{\theta}_i =
  \omega_i -
  \sum_{j=0}^n k_{ij} \sin(\theta_i - \theta_j)
  \quad\text{for } i = 0,\ldots,n,
\end{equation}
in which
$\theta_0,\ldots,\theta_n \in [0,2\pi)$
are the phases of the oscillators,
$\omega_i$'s are their natural frequencies
(i.e., their limit cycle frequencies),
where $k_{ij} = k_{ij}$ are constant coupling coefficients.
This model has since been called the Kuramoto model.

In the study of this model,
one important problem is the classification of
``frequency synchronization configurations'',
which correspond to the equilibria of \eqref{equ: generalized kuramoto}.
Although the equilibrium equations are not algebraic,
through a change of variables
and a relaxation to include complex configurations,
we can consider the algebraic synchronization equation
\cite{Baillieul1982},
given by
\begin{equation}\label{equ:kuramoto-rat}
    0 =
    f_i(\boldx) = \omega_i - \sum_{j=0}^n a_{ij}
    \left(
        \frac{x_i}{x_j} - \frac{x_j}{x_i}
    \right)
    \quad\text{for } i = 1,\dots,n
\end{equation}
where $a_{ij} = \frac{k_{ij}}{2}$ are complex constants,
$x_0 = 1$,
and $x_i = e^{\mathfrak{i} \theta_i}$.
The complex zero set of $(f_1,\ldots,f_n)$
captures the equilibria of \eqref{equ: generalized kuramoto}
which correspond to the synchronization configurations.
The problem of counting such configurations
thus relaxes to a root counting problem.
Since each $f_i$ in \eqref{equ:kuramoto-rat}
is a linear combination of $1$ and Laurent polynomials
$x_i x_j^{-1} - x_j x_i^{-1}$,
it is natural to view the generic root count for this system
as the intersection index $[\, L_1,\ldots,L_n \,]$
where
\begin{equation}\label{equ: node space}
    L_i = \Span_{\C} \{ 1 \} \cup \{ x_i x_j^{-1} - x_j x_i^{-1} \}_{j=0}^n
    \quad\text{for } i = 1,\ldots,n.
\end{equation}
Through a ``modified B\'ezout technique'',
Baillieul and Byrnes first computed this intersection index
(\cite[Theorem 4.1]{Baillieul1982}).
Their proof employed some rather deep results
in modern algebraic geometry.
In the following, we demonstrate the potential usefulness
of \Cref{thm:main} by providing a simple alternative proof.

\begin{proposition}
    Let $L_i$ be the vector spaces of Laurent polynomials
    as defined in \eqref{equ: node space}.
    Then $[\, L_1,\ldots,L_n\,]$ equal to
    the BKK bound of \eqref{equ:kuramoto-rat}.
\end{proposition}

\begin{proof}
    For each $i \in \{1,\dots,n\}$ and $j \in \{0,\dots,n\}$ we define
    \[
        P_{ij} = \frac{x_i}{x_j} - \frac{x_j}{x_i}.
    \]
    Then each $P_{ij}$ is a Laurent polynomial in the variables
    $\boldx = (x_1,\dots,x_n)$, and 
    $\Newt(P_{ij}) = \conv \{ \bolde_i - \bolde_j, \bolde_j - \bolde_i \}$
    where $\bolde_0 = \mathbf{0}$.
    So for each pair $(i,j)$, $\Newt(P_{ij})$ is a line segment through the origin.
    
    For each $i=1,\ldots,n$,
    we consider the vector space of rational functions
    \[
        L_i = \Span_{\C} \; \{ 1 \} \cup \{ P_{ij} \}_{j = 0}^n.
    \]
    Then functions in $L_i$ have no common zeros in $(\C^*)^n$
    It is easy to verify that $f_i \in L_i$ for each $i=1,\dots,n$.
    Therefore, the statement to be proved is equivalent to the claim that
    $[ L_1, \dots, L_n ]$ equals the BKK bound of the system $(f_1,\dots,f_n)$.
    
    By definition,
    \[
        \Newt(L_i) = \conv ( 
            \mathbf{0} \;\cup\; \{ \bolde_i - \bolde_j \,,\, \bolde_j - \bolde_i \}_{j=0}^n 
        ),
    \]
    and, in it, $\mathbf{0}$ is an interior point.
    For $n > 1$, $\Newt(L_i)$ is the convex hull of $n$ affinely independent
    line segments through the origin, and thus $\dim(\Newt(L_i)) = n$ for every $i$.
    Moreover, fixing $i$, for each $j=0,\dots,n$ and $j \ne i$, $\Newt(P_{ij})$
    is a line segment passing though an interior point, the origin, of $\Newt(L_i)$.
    Therefore, for each proper positive dimensional face $F$ of $\Newt(L_i)$,
    $F \cap \Newt(P_{ij})$ is either empty or a single point.
    By~\cref{thm:main}, the generic root count in $(\C^*)^n$,
    i.e., $[ L_1,\dots,L_n ]$ is exactly the BKK bound.
\end{proof}

\begin{remark}\label{rmk: kuramoto root count}
    Though Kuramoto originally only considered complete networks,
    in which every oscillator is influenced by every other oscillator,
    recent research activities have shifted toward sparse networks.
    Sparsity corresponds to the requirement that certain coefficients
    in \eqref{equ:kuramoto-rat} are zero.
    The generalizations of this result to sparse networks
    are developed in Refs.~\cite{ChenDavisMehta2018Counting,ChenKorchevskaiaLindberg2022Typical}
    using other tools.
\end{remark}

\section{Conclusions}\label{sec:conclusions}

In this short note, we
established a sufficient geometric condition under which 
the intersection index of a system of vector spaces of Laurent polynomials
equals its BKK bound.
This condition is stated purely in terms of the geometric information in 
the Newton polytopes of the polynomials involved
and can be checked easily using simple algorithms from convex geometry
without the information from corresponding tropical intersection.
It shows that certain algebraic relations among the coefficients 
have no effect on the exactness of the BKK bound.
The usefulness of this result is demonstrated through an application to
the algebraic Kuramoto equations --- 
a well studied family of equations used to model spontaneous synchronization phenomenon
in many fields. 
With this theorem, we easily established the equality between
the BKK bound and the more refined intersection index
of the algebraic Kuramoto equations,
even though this system has inherent algebraic relations among the coefficients.

\section*{Acknowledgments}

This note started as an attempt to answer
a question raised by Yizhong Li in 2018
during author's visit.
The author thank his hospitality
and comments.
The author's research is supported, in part,
by the National Science Foundation
under award no. DMS-1923099,
Auburn University at Montgomery Grant-in-aid program,
and Auburn University at Montgomery professional improvement leave award.




\bibliographystyle{abbrv}
\bibliography{nobody,bkk,chen,kuramoto}

\begin{thebibliography}{10}

\bibitem{Baillieul1982}
J.~Baillieul and C.~Byrnes.
\newblock {Geometric critical point analysis of lossless power system models}.
\newblock {\em IEEE Transactions on Circuits and Systems}, 29(11):724--737, nov
  1982.

\bibitem{Bernshtein1975Number}
D.~N. Bernshtein.
\newblock {The number of roots of a system of equations}.
\newblock {\em Functional Analysis and its Applications}, 9(3):183--185, 1975.

\bibitem{ChenDavisMehta2018Counting}
T.~Chen, R.~Davis, and D.~Mehta.
\newblock {Counting Equilibria of the Kuramoto Model Using Birationally
  Invariant Intersection Index}.
\newblock {\em SIAM Journal on Applied Algebra and Geometry}, 2(4):489--507,
  jan 2018.

\bibitem{ChenKorchevskaiaLindberg2022Typical}
T.~Chen, E.~Korchevskaia, and J.~Lindberg.
\newblock {On the typical and atypical solutions to the Kuramoto equations}.
\newblock {\em arXiv}, 2022.

\bibitem{KavehKhovanskii2008Convex}
K.~Kaveh and A.~Khovanskii.
\newblock {Convex bodies and algebraic equations on affine varieties}.
\newblock {\em arXiv preprint arXiv:0804.4095}, page~44, apr 2008.

\bibitem{KavehKhovanskii2012Newton}
K.~Kaveh and A.~Khovanskii.
\newblock {Newton-Okounkov bodies, semigroups of integral points, graded
  algebras and intersection theory}.
\newblock {\em Annals of Mathematics}, 176(2):925--978, sep 2012.

\bibitem{KavehKhovanskii2010Mixed}
K.~Kaveh and A.~G. Khovanskii.
\newblock {Mixed volume and an extension of theory of divisors}.
\newblock {\em Moscow Mathematical Journal}, 10(2):343--375, 2010.

\bibitem{Khovanskii1978Newton}
A.~G. Khovanskii.
\newblock {Newton polyhedra and the genus of complete intersections}.
\newblock {\em Functional Analysis and Its Applications}, 12(1):38--46, 1978.

\bibitem{Kushnirenko1976Polyedres}
A.~G. Kouchnirenko.
\newblock {Poly{\`{e}}dres de Newton et nombres de Milnor}.
\newblock {\em Inventiones Mathematicae}, 32(1):1--31, 1976.

\bibitem{Kuramoto1975Self}
Y.~Kuramoto.
\newblock {Self-entrainment of a population of coupled non-linear oscillators}.
\newblock Lecture Notes in Physics, pages 420--422. Springer Berlin Heidelberg,
  1975.

\bibitem{Kushnirenko1975Newton}
A.~G. Kushnirenko.
\newblock {A Newton polyhedron and the number of solutions of a system of k
  equations in k unknowns}.
\newblock {\em Usp. Math. Nauk}, 30:266--267, 1975.

\bibitem{Kushnirenko1976Newton}
A.~G. Kushnirenko.
\newblock {Newton polytopes and the Bezout theorem}.
\newblock {\em Functional Analysis and Its Applications}, 10(3):233--235, jul
  1976.

\bibitem{Minkowski1911Theorie}
H.~Minkowski.
\newblock {Theorie der konvexen Korper, insbesondere Begrundung ihres
  Oberflachenbegriffs}.
\newblock {\em Gesammelte Abhandlungen von Hermann Minkowski}, 2:131--229,
  1911.

\bibitem{Winfree1967Biological}
A.~T. Winfree.
\newblock {Biological rhythms and the behavior of populations of coupled
  oscillators}.
\newblock {\em Journal of Theoretical Biology}, 16(1):15--42, 1967.

\end{thebibliography}

\end{document}